\documentclass[11pt,reqno]{amsart}

\usepackage[papersize={8.5in,11in},width=6in,height=8in,centering]{geometry}
\usepackage{amsfonts}
\usepackage{amssymb}
\usepackage{amsthm}
\usepackage{amsmath}
\usepackage{graphicx}

\usepackage[all]{xy}

\theoremstyle{definition}
 \newtheorem{theorem}{\bf Theorem}[section]
 
 \newtheorem{lemma}[theorem]{Lemma}
 \newtheorem{corollary}[theorem]{Corollary}
 
\theoremstyle{definition}
 
 \newtheorem{remark}[theorem]{Remark}
 \newtheorem{definition}[theorem]{Definition}
 \newtheorem{proposition}[theorem]{Proposition}
\numberwithin{equation}{section}

\author[A. Marinkovi\'c]{Aleksandra Marinkovi\'c}
\address{Center for Mathematical Analysis, Geometry and Dynamical Systems, 
Mathematics Department, 
Instituto Superior T\'ecnico, Universidade de Lisboa 
Av. Rovisco Pais, 1049-001 Lisboa, Portugal }
\curraddr{}
\email{aleksperisic@yahoo.com}
\urladdr{}
\dedicatory{}
\title[]{Symplectic fillability of toric contact manifolds}

\begin{document}
\begin{abstract}
According to Lerman, compact connected toric contact 3-manifolds with a non-free toric action whose moment cone spans an angle greater than $\pi$ are overtwisted, thus non-fillable. In contrast,
we show that all compact connected toric contact manifolds in dimension greater than three are weakly symplectically fillable and most of them are strongly symplectically fillable. The proof is based on the Lerman's classification of toric contact manifolds and on our observation that the only contact manifolds in higher dimensions that admit free toric action are the cosphere bundle of $T^d, d\geq3$ $(T^d\times S^{d-1})$ and $T^2\times L_k,$ $k\in\mathbb{N},$ with the unique contact structure.
\end{abstract}

\maketitle
{\it keywords}: contact manifold, toric action, symplectic fillability.

\section{Introduction}
A \textbf{toric contact manifold} is a co-oriented contact manifold $(V^{2d-1},\xi)$ with an effective action of the torus $T^d=(\mathbb{R}/2\pi\mathbb{Z})^d,$ that preserves the contact structure $\xi$.
A complete classification of compact connected toric contact manifolds was done by Lerman \cite[Theorem 2.18.]{Lerman1}. 
 Contact $3$-manifolds that admit a free toric action are contactomorphic to $(T^3,\xi_k=\ker(\cos( k\theta) d\theta_1+\sin( k\theta) d\theta_2)), k\in\mathbb{N},$ while contact 3-manifolds that admit a non-free toric action are topologically $S^1\times S^2$ and lens spaces $L_{p/q},$ $p\in\mathbb{Z},$ $q\in\mathbb{Z}\backslash\{0\}$ with various toric contact structures. Higher dimensional contact manifolds with a non-free toric action are either of Reeb type or they are contactomorphic to $T^k\times S^{2d-k-1},k\in\mathbb{N}$  with the unique toric contact structure (see Section \ref{subsection non free}). Finally, higher dimensional contact manifolds with a free toric action are principal $T^{d}-$bundles over a sphere $S^{d-1}$ and each principal $T^{d}-$bundle over $S^{d-1}$ admits unique toric contact structure. These bundles are classified by $\pi_1(S^{d-2},\mathbb{Z}^d)$ and thus, if $d\neq3$ there is only a trivial bundle and if $d=3$ there are $\mathbb{Z}^3$ such bundles, each of them uniquely represented by the triple of integers. In Section \ref{subsection free} we show that all toric contact structures corresponding to the triples with the same greatest common divisor are contactomorphic and that
 all these contact structures can be obtained by the Bourgois construction \cite{Bourgeois} starting from the natural open book on the standard contact sphere.

\vskip1mm
In this article we examine the question of strong and weak symplectic fillability of toric contact manifolds. A symplectic fillability is an invariant of a contact structure and does not depend on the toric action. A compact symplectic manifold $(W^{2d},\omega)$ is called \textbf{a strong symplectic filling} of a co-oriented contact manifold $(V^{2d-1},\xi)$ if $V$ is a topological boundary of $W$, and if there is a vector field $X$ defined in a neighborhood of $V$ in $W$ such that $\omega(X,\cdot)|_{T V}$ is a positive contact form for $(V,\xi)$ and $L_X\omega=\omega.$
Note that the last condition implies that $\omega$ is exact near the boundary of $W,$ due to Cartan's formula
$L_X\omega=d\omega(X,\cdot,\cdot)+d(\omega(X,\cdot))=d(\omega(X,\cdot)).$ When the condition $L_X\omega=\omega$ holds on the whole manifold $W,$ then the vector field $X$ is called a Liouville vector field.
A compact symplectic 4-manifold $(W,\omega)$ is called \textbf{a weak symplectic filling} of a contact 3-manifold $(V,\xi)$ if $V$ is a topological boundary of $W$ and if $\omega|_{\xi}>0,$ that is, if $\omega$ gives a positive orientation of $\xi.$
A definition of weak fillability in dimensions greater than three was recently introduced by Massot, Niederkr$\mathrm{\ddot{u}}$ger and Wendl in \cite{MKW} and it generalizes this definition in dimension three.
A compact symplectic manifold $(W,\omega)$ is called \textbf{a weak symplectic filling} of a contact manifold $(V,\xi)$ if $V$ is a topological boundary of $W$, and if there is a positive contact form $\alpha$ on $V$ such that the orientation on $V$ given by $\alpha$ agrees with the boundary orientation of $W$ and $\alpha\wedge(\omega|_{\xi}+\tau d\alpha)^{d-1}>0,$ for all $\tau\geq0.$ Note that both definitions (strong and weak fillability) do not depend on the choice of a positive contact form for the same contact structure. A contact manifold $(V,\xi)$ is called \textbf{strongly (weakly) symplectically fillable} if it allows a strong (weak) symplectic filling.
A strong symplectic filling is also a weak symplectic filling. Indeed, if $(W,\omega)$ is a strong filling, then $\omega=d\alpha$ on $V,$ for some positive contact form $\alpha$ and thus $\alpha\wedge(\omega|_{\xi}+\tau d\alpha)^{d-1}=(1+\tau)^{d-1}\alpha\wedge d\alpha^d>0.$ In contrast, a weak filling may not be a strong filling. For example, $(T^3,\xi_k=\ker(\cos k\theta d\theta_1+\sin k\theta d\theta_2)), k>1$ is weakly fillable \cite{Giroux} but not strongly fillable \cite{Eliashberg2} and examples of weakly but not strongly fillable 5-dimensional contact manifolds are provided in \cite{MKW}. 
\vskip1mm
As proved by Eliashberg in \cite{Eliashberg0} an obstruction to fillability of contact 3-manifolds is overtwistedness. A contact structure $\xi$ on a 3-manifold $V$ is called \textbf{overtwisted} if there is a disc $\Delta\subset V$ that is tangent to $\xi$ along the boundary $\partial\Delta,$ that is $T_x\Delta=\xi_x,$ at every point $x\in\partial\Delta.$ The notion of overtwisted contact structures is recently generalized to higher dimensional contact manifolds by Borman, Eliashberg and Murphy \cite{bem}. They also show that an overtwisted contact structure is not (semi-positive) weakly fillable. 
\vskip1mm
We first summarise already known results on fillability of toric contact 3-manifolds:
\begin{theorem}\label{theorem 3-dim}
  A compact connected toric contact 3-manifold with a free toric action is strongly symplectically fillable if it is contactomorphic to $(T^3,\xi_1).$ Otherwise (when it is contactomorphic to $(T^3,\xi_k),$ for some $k>1$), it is only weakly fillable. A compact connected toric contact 3-manifold with a non-free toric action is strongly symplectically fillable if the corresponding moment cone spans an angle not greater than $\pi$. Otherwise, it is overtwisted.
 \end{theorem}
The main result in this article is the following theorem:
\begin{theorem}\label{theorem fillable} Any compact connected toric contact manifold of dimension greater than three is weakly symplectically fillable. Moreover, if the toric action is not free, or if the toric contact manifold is  a trivial principal $T^d$-bundle over $S^{d-1}$ (with a free toric action), then $V$ is strongly symplectically fillable. 
\end{theorem}
The proof is based on the Lerman's classification of toric contact manifolds and on the classification of contact structures on non-trivial $T^3$-bundles over $S^2$ that admit free toric action, done in Section \ref{subsection free}. For these  particular contact structures we are able to show only weak fillability (see Proposition \ref{proposition free is fillable}). We do not know if strong fillability result holds in these cases. 
\vskip1mm
As already mentioned, Borman, Eliashberg and Murphy in \cite{bem} introduced the notion of overtwisted contact structures in dimensions higher than three. Since overtwisted contact structures are not weakly fillable  we conclude:
\begin{corollary} There does not exist toric contact manifold in dimension greater than three with overtwisted contact structure.
\end{corollary}
    \vskip1mm
{\bf Acknowledgments.} I would like to thank Professors Miguel Abreu, Gustavo Granja, Klaus Niederkr$\mathrm{\ddot{u}}$ger, Milena Pabiniak and Branislav Prvulovi\'c for many useful disscussions. Moreover I thank the anonymous referees for their comments which have significantly improved an earlier version of this paper.
My research was supported by the Funda\c{c}\~ao para a Ci\^encia e a Tecnologia (FCT, Portugal) grant SFRH/BD/77639/2011 and project PTDC/MAT/117762/2010. 
\section{Toric contact manifolds} \label{section classification}
In this Section we recall the classification of compact connected toric contact manifolds done by Lerman and starting from this classification we completely determine all contact manifolds with a free toric action. We focus on the properties that are relevant to show symplectic fillability.
\vskip1mm
To any toric contact manifold $(V^{2d-1},\xi)$ and a $T^d$-invariant contact form $\alpha$ one can associate a $T^d$-invariant map, called a moment map,  $\mu_{\alpha}=(\mu_1,\ldots,\mu_d):V\rightarrow\mathbb{R}^d,$ defined by 
$\mu_k=\alpha(X_k),$ $k=1,\ldots d,$ where, for every $p\in V,$ $X_k(p)=(\frac{d}{dt}|_{t=0}exp\hskip1mm (0,\ldots,t,\ldots 0))\ast p.$ Note that an invariant contact form always exists. Indeed, if $\alpha$ is not invariant we obtain an invariant contact form $\alpha_{\textrm{inv}}$ by averaging, that is, $\alpha_{\textrm{inv}}=\int_{t\in T^d}t^*\alpha.$ \textbf{The moment cone} of a toric contact manifold $(V,\xi)$ is defined to be the cone over $\mu_{\alpha}(V)$, or, equivalently, it is the union of the origin and the moment map image of $(V\times\mathbb{R}_+,d(e^t\alpha_{\textrm{inv}})),$ the symplectization of $V,$ that is a toric symplectic manifold. Note that the definition of the moment cone does not depend on the choice of an invariant contact form.
\begin{definition}\label{def isomorphic}
Two toric contact manifolds $(V_i^{2d-1},\xi_i), i=1,2$ are \textbf{equivalent} if there is a contactomorphism $\varphi:V_1\mapsto V_2$ and an isomorphism $\lambda:T^d\mapsto T^d$ such that $\varphi(t\ast_1 p)=\lambda(t)\ast_2\varphi(p),$ for evety $t\in T^d$ and every $p\in V_1,$ where $\ast_1$ and $\ast_2$ denote toric actions on $V_1$ and $V_2$ respectively.
  \end{definition}
  If two toric contact manifolds are equivalent, then the corresponding moment cones differ by a $SL(d,\mathbb{Z})$ transformation. On the other hand, if moment cones differ by a $SL(d,\mathbb{Z})$ transformation and they uniquely determine toric contact manifolds, then these toric contact manifolds are equivalent.
\subsection{Toric contact manifolds with non-free toric action} \label{subsection non free}
 Let $(V^{2d-1},\xi)$ be a compact toric contact manifold with a non-free toric action.
 \vskip1mm
  $\bullet$ If $\dim V=3,$ then $V$ is diffeomorphic to a lens space $L_{p/q},$ $p\in\mathbb{Z},$ $q\in\mathbb{Z}\backslash\{0\}$ or $S^1\times S^2$  \cite[Theorem 2.18.(ii)]{Lerman1}. As a toric contact manifold $V$ is classified by two real numbers $t_1$ and $t_2$ such that  $t_1\in[0,2\pi),$ $t_1<t_2$ and $\tan t_i\in\mathbb{Q}$ or $\cos t_i=0,$ $i=1,2.$
If  $t_2-t_1\geq2\pi,$ then the moment cone of $V$ is the whole space $\mathbb{R}^2$. If $t_2-t_1<2\pi,$ then
the moment cone of $V$ has two facets $\{(s\cos t_i,s\sin t_i)\in\mathbb{R}^2|s\geq0\},$ $i=1,2,$ with the angle between facets $t_2-t_1.$ 
  In particular, when $t_2-t_1<\pi,$ then $V$ is of \textbf{Reeb type} (see below). When $t_2-t_1=\pi,$ then $V$ is contactomorphic to $(S^1\times S^2, \ker(xd\theta+\frac{i}{4}(zd\overline{z}-\overline{z}dz))),$ where $S^2=\{(z,x)\in\mathbb{C}\times\mathbb{R}|\hskip1mm |z|^2+x^2=1\}$. When $t_2-t_1>\pi,$ then the contact structure $\xi$ is overtwisted (see the Proof of Theorem 1.5 in \cite{Lerman1}). First examples of overtwisted toric contact manifolds are constructed by Lerman in \cite{Lerman2} using the method of contact cut.
  \vskip1mm
$\bullet$ If $\dim V>3,$ then $(V,\xi)$ is uniquely determined by the moment cone \cite[Theorem 2.18.(iii)]{Lerman1} and the cone is always convex \cite[Theorem 4.2.]{Lerman1}.
  Moreover:
\begin{enumerate} 
\item If the cone is \textbf{strictly convex}, meaning it does not contain any linear subspace of a positive dimension, then $V$ is a toric contact manifold of \textbf{Reeb type} \cite{BoyerGalicki} and a prequantization space. Here is the sketch of the proof, for more details see \cite[Proposition 2.19.]{AbreuMacarini}, \cite[Lemma 3.7.]{Lerman3}, \cite[Theorem 2.7.]{Boyer}. We remark that all properties of toric contact manifolds of Reeb type also hold in 3-dimensional case. According to the definition, the moment cone for $V$ withouth the origin is a moment map image of  the symplectization $W$ of $V,$ that is a (non-compact) toric symplectic manifold. Take any vector $R\in\mathbb{Z}^d$ given as a positive linear combination of inward normal vectors to the facets of the cone. The vector $R$ defines a $S^1$-subaction of the toric action on $W$. The symplectic reduction of $W$ with respect to that $S^1$-action is a toric symplectic orbifold $(M,\omega),$ whose moment polytope is the intersection of the moment cone and one hyperplane perpendicular to $R$. Changing the hyperplane corresponds to rescaling the symplectic form $\omega.$ The prequantization of $M,$ that is, the principal $S^1$-bundle with the first chern class $\frac{1}{2\pi}[\omega],$ is $V$ and the connection 1-form is $\alpha$ such that $d\alpha=\pi^*\omega,$
where $\pi:V\to M$ is the projection. The 1-form $\alpha$ is a contact form for $\xi$ and the $S^1$-bundle action on $V$ is precisely the action given by the Reeb vector field $R_{\alpha}$. The toric action on $M$ lifts to a torus action on $V$ that commutes with this $S^1$-action. A direct sum of these actions is a toric action on $V.$ Since the Reeb vector field $R_{\alpha}$ generates $S^1$-subaction of the toric action, by definition introduced by Boyer and Galicki, $V$ is of Reeb type. 
\item If the cone is \textbf{not strictly convex} and $k>0$ denotes the dimension of the linear subspace contained in the cone, then $k<d$ \cite[Lemma 4.5.]{Lerman1}, i.e. the cone is not equal to the whole space. Such a cone is isomorphic (there is $SL(d,\mathbb{Z})$ transformation) to the cone $C=\{x_1,\ldots,x_{d-k}\geq0\}\subset\mathbb{R}^d.$ The toric contact manifold corresponding to the cone $C$ is $T^k\times S^{2d-k-1},$ with the following contact structure. Using the coordinates
$$T^{k}\times S^{2d-k-1}=\Big\{(e^{i\theta_1},\ldots,e^{ i\theta_{k}},x_1,\ldots,x_k,z_1,\ldots ,z_{d-k})\in
T^{k}\times\mathbb{R}^k\times\mathbb{C}^{d-k}\hskip1mm|$$
$$\hskip1mm \sum_{l=1}^k|x_l|^2+\sum_{j=1}^{d-k}|z_j|^2=1\Big\},$$
 the contact structure is given as the kernel of the following contact form
$$\beta_{k}=\sum_{l=1}^{k}x_ld\theta_{l}+\frac{i}{4}\sum_{j=1}^{d-k}(z_jd\overline{z}_j-\overline{z}_jdz_j)$$
and the toric
 $T^{d}$-action on $T^k\times S^{2d-k-1}$ is given by
$$(s_1,\ldots,s_k,t_1,\ldots,t_{d-k})\ast (e^{i\theta_1},\ldots,e^{ i\theta_k},x_1,\ldots,x_k,z_1,\ldots ,z_{d-k})\longmapsto$$
$$(s_1e^{ i\theta_1},\ldots,s_ke^{i\theta_k},x_1,\ldots,x_k,t_1 z_1,\ldots, t_{d-k} z_{d-k}).$$
According to classification theorem by Lerman \cite[Theorem 2.18.(iii)]{Lerman1} it follows that $(V,\xi)$ is equivalent to $(T^{k}\times S^{2d-k-1},\ker\beta_k).$
\end{enumerate}
\subsection{Toric contact manifolds with a free toric action}\label{subsection free} 
Let $(V^{2d-1},\xi)$ be a compact toric contact manifold with a free toric action. Then the moment cone is the whole space $\mathbb{R}^d$. Indeed, consider the normalized invariant contact form $\alpha,$ that is, a contact form with $||\mu_{\alpha}(p)||=1,$ for every $p\in V,$ where $||\cdot||$ is the norm on $\mathbb{R}^d$ given by the standard inner product. This contact form can be defined by $\alpha(p)=\frac{1}{||\mu_{\beta}(p)||}\beta(p),$ where $\beta$ is any invariant contact form. Since $\mu_{\beta}(p)\neq0,$ for every $p\in V$ \cite[Lemma 2.12.]{Lerman1}, it follows that $\alpha$ is well defined. According to \cite[Corollary 4.7.]{Lerman1} the moment map $\mu_{\alpha}:V\rightarrow S^{d-1}\subset \mathbb{R}^d$ is a submersion and in particular surjective. Thus the moment cone is equal to $\mathbb{R}^d.$ (Note that toric contact manifolds with a non-free toric action whose moment cone is equal to the whole space exist only in dimension three.)
\vskip1mm
 $\bullet$ If $\dim V=3$ and $k\in\mathbb{N}$ is the number of connected components of the fibers of the moment map, then $V$ is contactomorphic to $(T^3,\xi_k=\ker(\cos k\theta d\theta_1+\sin k\theta d\theta_2))$ \cite[Theorem 2.18.(i)]{Lerman1}. Here $\theta_1,\theta_2$ and $\theta$ denote the coordinates on $T^3=(\mathbb{R}/2\pi\mathbb{Z})^3$.
 \vskip1mm
 $\bullet$ If $\dim V>3,$ then the fibers of the moment map are connected, thus $V$ is a principal $T^d$-bundle over $S^{d-1}.$ Moreover, each such bundle admits unique toric contact structure \cite[Theorem 2.18.(iii)]{Lerman1}. These contact structures are particular examples of the contact structures constructed by Lutz \cite{Lutz}. The toric contact manifold corresponding to the trivial bundle is the cosphere bundle of $T^d,$ that is
  $(T^d\times S^{d-1},\ker(\sum_{i=1}^dx_id\theta_i))$ with the unique toric action  \cite[Theorem 1.3.]{Lerman1} given by
$$(t_1,\ldots,t_d)\ast (e^{i\theta_1},\ldots,e^{ i\theta_d},x_1,\ldots,x_d)\longmapsto
(t_1e^{ i\theta_1},\ldots,t_de^{i\theta_d},x_1,\ldots,x_d).$$
Non-trivial principal $T^d$-bundles over $S^{d-1}$ exist only when $d=3.$ Indeed, since a $T^d$-bundle over a sphere $S^{d-1}$ can be trivialized over two open sets: upper and lower hemisphere, it follows that it is determined by only one transition map (up to homotopy), the map from the intersection of these sets, that is, from equator $S^{d-2},$ to the torus $T^d.$ Thus, these bundles are classified by $\pi_{d-2}(T^d)=(\pi_{d-2}(S^1))^d.$ Since $\pi_{d-2}(S^{1}) = 0$
for $d-2>1$ and $\pi_1(S^1)=\mathbb{Z}$ it follows that when $d>3$ there is only the trivial principal $T^d$-bundle over $S^{d-1}$ while
when $d=3,$ there are $\mathbb{Z}^3$ such bundles, each of them represented by the triple of integers $(k_1,k_2,k_3).$
Due to Lerman, each triple represents unique toric contact manifold with a free toric action.
\vskip1mm
We now want to show that all toric contact manifolds corresponding to the triples with the same greatest common divisor are contactomorphic.

\begin{lemma}\label{total space} The total space of principal $T^3$-bundles over $S^2$ that are represented by triples with the greatest common divisor equal to $k\in\mathbb{Z}\backslash\{0\}$ is $T^2\times L_k$ and the corresponding $T^3$-actions  differ only by a reparametrization of $T^3.$
\end{lemma}
\begin{proof}
Each $T^3$-bundle over the sphere $S^2$ is uniquely determined by only one transition map $g:S^1\mapsto T^3.$ The bundle represented by the triple $(k_1,k_2,k_3)$ has the transition map $g_{(k_1,k_2,k_3)}(e^{i\theta})=(e^{k_1\theta},e^{k_2\theta},e^{k_3\theta}).$ Thus, in order to show that two bundles with transition maps $g_{(k_1,k_2,k_3)}$ and $g'_{(k'_1,k'_2,k'_3)}$ are isomorphic, it is enough to find a map $\lambda:T^3\mapsto T^3$ such that $\lambda\circ g_{(k_1,k_2,k_3)}=g'_{(k'_1,k'_2,k'_3)},$ that is, to find a reparametrization of the torus $T^3.$ In terms of integer triples, that means to find a $SL(3,\mathbb{Z})$ matrix that
     sends $(k_1,k_2,k_3)$ to $(k_1',k_2',k_3').$
     \vskip1mm
     Let $(k_1,k_2,k_3)\in\mathbb{Z}^3$ be a primitive vector, that is, a triple with $GCD(k_1,k_2,k_3)=1$ (If $k_i=0,$ for some $i\in\{1,2,3\},$ then we take the greatest common divisor of the remaining ones).
 The vector $(k_1,k_2,k_3)\in\mathbb{Z}^3$ can be completed to a $\mathbb{Z}$-basis of $\mathbb{Z}^3.$ If $(k_1,k_2,k_3),e_1,e_2$ are such a basis, then the inverse of the matrix with columns given by $(k_1,k_2,k_3),e_1,e_2$  is an integer matrix with determinant equal to 1 that sends  $(k_1,k_2,k_3)$ to $(1,0,0).$ The same matrix sends $(k_1,k_2,k_3),$ with $GCD(k_1,k_2,k_3)=k,$ to $(k,0,0).$ This matrix defines a reparametrization of the torus $T^3$ that shows that bundles represented by $(k,0,0)$ and $(k_1,k_2,k_3)$ with $GCD(k_1,k_2,k_3)=k$ are isomorphic.
   \vskip1mm
 Let us now explain the total space of these bundles. 
      In general,
all principal $G-$bundles over a CW-complex $X,$ up to isomorphism of the bundles, are in bijection with $[X,BG]$, homotopy classes of all maps $f:X\rightarrow BG,$ where $BG$ denotes the classifying space for $G.$ Precisely, every such bundle is isomorphic to the bundle  $f^*\gamma$,  the pull back of the universal bundle $\gamma$ corresponding to $G,$ for some map $f:X\rightarrow BG.$ In particular, every principal $G=G_1\times G_2-$bundle is isomorphic to the bundle $f^*(\gamma_1\times \gamma_2),$ for some function $f:X\rightarrow (B(G_1\times G_2)\cong BG_1\times BG_2),$ where $\gamma_i$ is a universal principal $G_i$-bundle, $i=1,2.$  Since $f=(f_1,f_2)\circ\Delta,$ where $f_i:X\rightarrow BG_i$ and $\Delta:X\rightarrow X\times X$ is the diagonal map, it follows that $f^*(\gamma_1\times \gamma_2)=\Delta^*((f_1,f_2)^*(\gamma_1\times \gamma_2))=\Delta^*(f_1^*\gamma_1,f_2^*\gamma_2)=f_1^*\gamma_1\times_X f_2^*\gamma_2.$ This means that
  there is a bijection between principal $G_1\times G_2-$bundles over $X$ and a fibre sum of principal $G_1-$bundles over $X$ and principal $G_2-$bundles over $X$. 
\vskip1mm
In particular, there is a bijection between principal $T^3-$bundles over $S^2$ and a fibre sum of three principal $S^1-$bundle over $S^2.$ Zero in the triple $(k_1,k_2,k_3)$ represents a trivial $S^1$-bundle over $S^2$ while $k\in\mathbb{Z}\backslash\{0\}$ represents a non-trivial principal $S^1-$bundle over $S^2,$ the lens space $L_k,$ $k\in\mathbb{Z}\backslash\{0\},$ where $L_k=S^3/_{(z_1,z_2)\sim(e^{\frac{2\pi i}{k}}z_1,e^{\frac{2\pi i}{k}}z_2)}.$ Thus, the total space of the bundle represented by $(k,0,0), k\in\mathbb{Z}\backslash\{0\},$ is $T^2\times L_k.$
\vskip1mm
The free $T^3$-action on $T^2\times L_k$ represented by the triple $(k,0,0)$ is given by
\begin{equation}\label{(1,0,0)}
(t_1,t_2,t_3)\ast(e^{ i\theta_1},e^{ i\theta_2},[z_1,z_2]_k)\rightarrow(t_1e^{ i\theta_1},t_2e^{ i\theta_2},t_3\ast[z_1,z_2]_k),
\end{equation}
for every $(t_1,t_2,t_3)\in T^3$ and every $(e^{ i\theta_1},e^{ i\theta_2},[z_1,z_2]_k)\in T^2\times L_k,$
where  $[z_1,z_2]_k$ denotes a point in $L_k$ and the free $S^1=(\mathbb{R}/2\pi\mathbb{Z})-$action on $L_k$ is given by $t\ast[z_1,z_2]_k\rightarrow t^{\frac{1}{k}}[z_1,z_2]_k=[t^{\frac{1}{k}}z_1,t^{\frac{1}{k}}z_2]_k.$ Thus,
the $T^3$-action on the bundle $T^2\times L_k$ represented by the triple $(k_1,k_2,k_3)$ is given by 
$$\lambda(t_1,t_2,t_3)\ast(e^{ i\theta_1},e^{ i\theta_2},[z_1,z_2]_k),$$
 for every $(t_1,t_2,t_3)\in T^3$ and every 
    $(e^{ i\theta_1},e^{ i\theta_2},[z_1,z_2]_k)\in T^2\times L_k,$ where $\lambda:T^3\mapsto T^3$ is given by the matrix that sends $(k_1,k_2,k_3)$ to $(k,0,0).$ 
That means that $T^3$-actions represented by the triples with the same greatest common divisor differ only by a reparametrization of $T^3.$
\end{proof}
Let us find an invariant contact structure on these bundles. We start from the map $f(z_1,z_2)=z_1^2+z_2^2,$ that defines a compatible open book on $(S^3,\xi_{st}=\ker\frac{i}{4}\sum_{j=1}^2(z_jd\overline{z}_j-\overline{z}_jdz_j))$ with the binding $f^{-1}(0)\cap S^3$ and the fibration $\pi_{\theta}(z_1,z_2)=\frac{f(z_1,z_2)}{|f(z_1,z_2)|}.$ We now want to use the method by Bourgeois \cite{Bourgeois} to show that $\textrm{Re}f d\theta_1+\textrm{Im}fd\theta_2+\frac{i}{4}\sum_{j=1}^2(z_jd\overline{z}_j-\overline{z}_jdz_j)$ is a contact form on $T^2\times S^3.$ First, the binding $f^{-1}(0)\cap S^3$ consists of two connected components, circles: $\{(z_1,iz_1)\in S^3\}$ and $\{(z_1,-iz_1)\in S^3\}.$ Next, the solid tori $\{|iz_1-z_2|<|iz_1+z_2|\}$ and $\{|iz_1-z_2|>|iz_1+z_2|\}$ are tubular neighborhoods inside $S^3$ of these circles, respectively. Finally, $S^3$ is the union of these solid tori and the $T^2$-torus $\{|iz_1-z_2|=|iz_1+z_2|\}.$ Now consider a smooth function $\rho(z_1,z_2)=|f(z_1,z_2)|.$ This function is a distance function from the binding in the small neighborhood of the binding and is constant (equal to the square root of the radius of the sphere) on the torus  $\{|iz_1-z_2|=|iz_1+z_2|\},$ i.e. outside of the tubular neighborhood of the binding. If we consider the sphere of a radius $\epsilon,$ for a sufficiently small $\epsilon>0,$ then, due to Bourgeois \cite{Bourgeois}, the map $\rho\pi_{\theta}=f$ defines a contact form on $T^2\times S^3$ by $\textrm{Re}f d\theta_1+\textrm{Im}fd\theta_2+\frac{i}{4}\sum_{j=1}^2(z_jd\overline{z}_j-\overline{z}_jdz_j).$ 
A pull back of this contact form under the diffeomorphism $(e^{ i\theta_1},e^{ i\theta_2},z_1,z_2)\mapsto
   (e^{ i\theta_1},e^{ i\theta_2},\frac{1}{\sqrt{2}}(z_1+i\overline{z}_2,iz_1+\overline{z}_2))$ is a contact form
\begin{equation}\label{alpha}
\alpha=\alpha_1=i(z_1\overline{z}_2-\overline{z}_1z_2)d\theta_1+(z_1\overline{z}_2+
\overline{z}_1z_2)d\theta_2+\frac{i}{4}(
z_1d\overline{z}_1-\overline{z}_1dz_1-(z_2d\overline{z}_2-\overline{z}_2dz_2))
\end{equation}
and $\alpha$ is invariant under the toric $T^3$-action represented by the triple $(1,0,0)$ (see (\ref{(1,0,0)})).
 In particular, $\alpha$ is invariant under the diagonal $\mathbb{Z}_k$-action on $S^3$-factor of $T^2\times S^3$, so it descends to a contact form $\alpha_k$ on the quotient $T^2\times L_k$ and the contact form $\alpha_k$ is invariant under the $T^3$-action represented by $(k,0,0)$ (\ref{(1,0,0)}). Since the toric actions $(k,0,0)$ and $(k_1,k_2,k_3),$ for any choice of $k_1,k_2,k_3\in\mathbb{Z}$ with $GCD(k_1,k_2,k_3)=k$ differ by a reparametrization of the torus, it follows that $\alpha_k$ is invariant under toric action represented by any triple $(k_1,k_2,k_3)$ with $GCD(k_1,k_2,k_3)=k$. Thus, we conclude $(T^2\times L_k,\ker\alpha_k)$ is a contact manifold with free toric actions, for any choice of $k_1,k_2,k_3$ with $GCD(k_1,k_2,k_3)=k$. \\
 We proved the following Theorem:
 \begin{theorem}\label{theorem determine} Contact manifolds with a free toric action that correspond to non-trivial $T^3$-bundles over $S^2$ are of the form $(T^2\times L_k, \ker\alpha_k), k\in\mathbb{N}$ with the unique toric action given by (\ref{(1,0,0)}). More precisely, all toric contact manifolds represented by the triples with the greatest common divisor equal to $k$ are contactomorphic to $(T^2\times L_k,\ker \alpha_k)$ and the toric actions are equivalent, i.e. they differ by a reparametrization of the torus. 
 \end{theorem}
\begin{remark} As explained, for any $k\in\mathbb{N},$ the contact manifold $(T^2\times L_k,\ker\alpha_k)$ admits unique toric action, up to equivalence (Definition \ref{def isomorphic}). However, it admits infinitely many non-equivariant toric actions. Two toric actions on $(V,\xi)$ are \textbf{equivariant} if there is a contactomorphism $\varphi$ of $(V,\xi)$ such that $\varphi(t\ast_1 p)=t\ast_2\varphi(p),$ for every $p\in V$ and every torus element $t.$ That is, in contrast to the equivalence of toric actions (see Definition \ref{def isomorphic}), the reparametrization of the torus is not allowed. Assume that two toric actions represented by $(k_1,k_2,k_3)$ and $(k'_1,k'_2,k'_3)$ are equivariant. Then $\varphi$ is a bundle diffeomorphism, so it induces a diffeomorphism of basis $f: S^2\rightarrow S^2.$ If $\gamma_{(k_1,k_2,k_3)}$ and $\gamma_{(k'_1,k'_2,k'_3)}$ denote these bundles, then the bundles $f^*\gamma_{(k'_1,k'_2,k'_3)}$ and $\gamma_{(k_1,k_2,k_3)}$ are isomorphic ($f^*\gamma_{(k'_1,k'_2,k'_3)}=\gamma_{(k_1,k_2,k_3)}).$ Since the group of diffeomorphisms of $S^2$ is homotopy equivalent to the orthogonal group $O(3)$, and $O(3)$ has two connected components, $SO(3)$ and $-SO(3)$, it follows that $f$ belongs to one of these two subgroups. If $f\in SO(3),$ then $f$ is homotopic to the identity map $id$ on $S^2$ and thus 
  $$\gamma_{(k_1,k_2,k_3)}=f^*\gamma_{(k'_1,k'_2,k'_3)}=id^*\gamma_{(k'_1,k'_2,k'_3)}=\gamma_{(k'_1,k'_2,k'_3)}.$$
  On the other hand, if $f\in -SO(3),$ then $f$ is homotopic to the map $j(x,y,z)=(-x,-y,-z)$ on $S^2$ 
  and thus 
  $$\gamma_{(k_1,k_2,k_3)}=f^*\gamma_{(k'_1,k'_2,k'_3)}=j^*\gamma_{(k'_1,k'_2,k'_3)}=\gamma_{(-k'_1,-k'_2,-k'_3)}.$$
  We conclude that if $(k_1,k_2,k_3) \neq \pm (k_1',k_2',k_3'),$ then the actions $(k_1,k_2,k_3)$ and $(k'_1,k'_2,k'_3)$ are not equivariant. The choice of such triples with $GCD$ equal to $k$ is infinite.
\end{remark}
 \section{Weak and strong symplectic fillability}
 In this Section we examine the question of weak and strong symplectic fillability of compact connected toric contact manifolds using the classification described in the previous section.
 Note that the fillability property does not depend on the toric action, it only depends on the contact structure. 
 \vskip1mm
 As explained in Section \ref{subsection non free}.(1) every toric contact manifold of Reeb type is a prequantization of some toric symplectic orbifold (this also includes 3-dimensional case). It follows then that it is fillable by a symplectic orbifold. K. Niederkr$\mathrm{\ddot{u}}$ger and F. Pasquotto in \cite{NP} explained the resolution of the singular point in that orbifold and proved:
 \begin{theorem}\label{important theorem}(\cite[Proposition 4.4]{NP}) A contact manifold that is a prequantization of a symplectic orbifold is strongly symplectically fillable.
\end{theorem}
We conclude:
\begin{corollary}\label{corollary non free is fillable}  Any toric contact manifold of Reeb type is strongly symplectically fillable.
\end{corollary}
\begin{remark} From Corollary \ref{corollary non free is fillable} we see that a contact manifold with an overtwisted toric contact structure cannot be of Reeb type. 
Therefore the class of toric contact manifolds that are not of Reeb type, with a non-free toric action, is much larger in dimension $3$ than in higher dimensions, where there is only $T^k\times S^{2d-k-1},$ with the unique toric contact structure.
\end{remark}
\begin{proposition}\label{proposition non free is fillable} $(T^k\times S^{2d-k-1},\beta_{k}=\sum_{l=1}^{k}x_ld\theta_{l}+\frac{i}{4}\sum_{j=1}^{d-k}(z_jd\overline{z}_j-\overline{z}_jdz_j)), d\geq k\geq1$ is strongly symplectically fillable.
\end{proposition}
\begin{proof} The strong symplectic filling is
$(T^k\times D^{2d-k},\omega,X)$ where
$$\omega=\sum_{l=1}^{k}dx_l\wedge d\theta_{l}+\frac{i}{2}\sum_{j=1}^{d-k}dz_j\wedge d\overline{z}_j\hskip2mm\textrm{and}\hskip2mm X=\sum_{l=1}^{k}x_l\frac{\partial}{\partial x_l}+\frac{1}{2}\sum_{j=1}^{d-k}(z_j\frac{\partial}{\partial z_j}+\overline{z}_j\frac{\partial}{\partial \overline{z}_j}).$$
\end{proof}
\begin{remark} For the contact manifolds from previous Proposition a stronger notion of fillability holds. A complex manifold $(M, J)$ is a \textbf{Stein manifold} if it admits a smooth non-negative proper function $f$ such that $d(-df\circ J)$ is a symplectic form on $M,$ i.e. a plurisubharmonic function. A function $f(z)=|z|^2$ is a plurisubharmonic function on
 $(\mathbb{C}^m,\iota),$ thus a complex space is a Stein manifold. 
If $(M, J)$ is a Stein manifold, then $(M\times\mathbb{C}^m,J\times\iota)$ is also a Stein manifold, called \textbf{m-subcritical Stein manifold}. More precisely, if $f$ is a plurisubharmonic on $M,$ then $\widetilde{f}=f+|z|^2$ is a plurisubharmonic on $M\times\mathbb{C}^m.$
 A contact manifold $(V,\xi)$ is (m-subcritical) Stein fillable if there is (m-subcritical) Stein manifold $(M,J)$ such that $V$ is a regular level of a plurisubharmonic function $f,$ the vector field $\textrm{grad}\hskip1mm f$ points outward $V$ and $-df\circ J$ is a positive contact form for $\xi.$ Stein fillability implies strong fillability. Indeed, for $(V=f^{-1}(c),\ker(-df\circ J))$  the strong filling is ($W=f^{-1}[0,c], d(-df\circ J)$). The contact manifold from the previous Proposition is for $d=k$ Stein fillable (being a cosphere bundle of $T^k$) and for $d>k$ it is $(d-k)$-subcritical Stein fillable.
\end{remark}
  For the purpose of this note we prove the following Lemma in dimension $5$, but, as mentioned in \cite{MKW} it holds in higher dimensions as well.

\begin{lemma}\label{general lemma} If $(V,\xi=\ker\alpha')$ is a weakly fillable contact 3-manifold, then a contact manifold $(T^2\times V,\ker(\alpha=f_1d\theta_1+f_2d\theta_2+\alpha'))$ is weakly fillable, for any choice of smooth functions $f_i:V\rightarrow\mathbb{R}, i=1,2$ that makes $\alpha$ a contact form.
\end{lemma}
\begin{proof} Choose an orientation on $V$ such that $\alpha'$ is a positive contact form, i.e. $\alpha'\wedge d\alpha'>0.$ Let $(W,\omega)$ be a weak filling of $(V,\xi=\ker\alpha')$. It then holds $\alpha'\wedge(\omega+\tau d\alpha')>0,$ for all constants $\tau\geq0,$ and, in particular $\alpha'\wedge \omega>0.$ Next, we fix the orientation on $T^2\times V,$ given by a contact form $\alpha,$
that is $\alpha\wedge d\alpha^2>0.$
Note that $\alpha'\wedge d\alpha'\wedge\textrm{vol}_{T^2}$ is also a volume form on $T^2\times V$. Assume $$\alpha'\wedge d\alpha'\wedge d\theta_1\wedge d\theta_2>0,$$
i.e. assume $\alpha'\wedge d\alpha'\wedge d\theta_1\wedge d\theta_2$ induces the same orientation as $\alpha\wedge d\alpha^2$.
 Since $\alpha'\wedge \omega$ induces the same orientation on $V$ as $\alpha'\wedge d\alpha'$ also it holds
 $$\alpha'\wedge \omega\wedge d\theta_1\wedge d\theta_2>0.$$
 Now, instead of $\alpha,$ for any $t>0$ consider 1-form
$$\alpha_t=t(f_1d\theta_1+f_2d\theta_2)+\alpha'.$$
Since $\alpha_t\wedge d\alpha_t\wedge d\alpha_t=t^2\alpha\wedge d\alpha\wedge d\alpha$ it follows that $\alpha_t,$ for any $t>0,$ is a contact form on $T^2\times V$ inducing the same orientation as $\alpha=\alpha_1,$ that is, it  holds
$$\alpha_t\wedge d\alpha_t^2>0.$$
Due to Gray stability, for all $t>0,$ contact structures $\ker\alpha_t$ are contactomorphic. Since symplectic fillability is an invariant of contact structures, in order to show that $(T^2\times V,\ker\alpha)$ is weakly fillable it is enough to show weak fillability of $(T^2\times V,\ker\alpha_t)$ for some $t>0.$ Due to the definition of weak fillability,  we have to show
 $P_t(\tau)>0$ for all $\tau\geq0$ and for some fixed $t>0$ where
$$
P_t(\tau)=\alpha_t\wedge(\omega+ d\theta_1\wedge d\theta_2+\tau d\alpha_t)^{\wedge2}.
$$
We compute
$$P_t(\tau)=\tau^2\alpha_t\wedge d\alpha_t^2+2\alpha_t\wedge (\omega+\tau d\alpha_t)\wedge d\theta_1\wedge d\theta_2+2\tau\alpha_t\wedge\omega\wedge d\alpha_t.$$
Note that for all $t>0$ it holds
$$P_t(0)=2\alpha'\wedge \omega\wedge d\theta_1\wedge d\theta_2>0.$$
In order to show that for some small $t>0$ it holds $P_t(\tau)>0$ for all $\tau\geq0,$ it is enough to show that the function
$P_t(\tau)$ is increasing, since $P_t(0)>0.$ So, we want to show that the first derivative (with respect to $\tau$) of $P_t(\tau)$ is positive (for some small fixed $t>0$). It follows that
$$\frac{d}{d\tau}P_t(\tau)=2\tau\alpha_t\wedge d\alpha_t^2+2\alpha_t\wedge  d\alpha_t\wedge d\theta_1\wedge d\theta_2+2\alpha_t\wedge\omega\wedge d\alpha_t=$$
$$2\tau\alpha_t\wedge d\alpha_t^2+2\alpha'\wedge  d\alpha'\wedge d\theta_1\wedge d\theta_2
+2t^2 (f_2df_1-f_1df_2)\wedge \omega\wedge d\theta_1\wedge d\theta_2.$$
The first summand in $\frac{d}{d\tau}P_t$ is positive for all $t>0$ and all $\tau>0$. The second summand is a positive constant (doesn't depend on $\tau$).  Let us choose $t>0$ small enough such that
$$2\alpha'\wedge  d\alpha'\wedge d\theta_1\wedge d\theta_2+t^2 (f_2df_1-f_1df_2)\wedge \omega\wedge d\theta_1\wedge d\theta_2>0.$$
 For such a small $t>0,$ it follows that $\frac{d}{d\tau}P_t$ is positive, thus $P_t$ is monotone, increasing function.\\
Note that $t>0$ depends on the point $p\in T^2\times V$ in which we compute $2\alpha'\wedge  d\alpha'\wedge d\theta_1\wedge d\theta_2+t^2 (f_2df_1-f_1df_2)\wedge \omega\wedge d\theta_1\wedge d\theta_2.$ If $V$ is compact we can find $t$ that will not depend on the point, i.e we can find $t$ such that 
$2\alpha'\wedge  d\alpha'\wedge d\theta_1\wedge d\theta_2+t^2 (f_2df_1-f_1df_2)\wedge \omega\wedge d\theta_1\wedge d\theta_2$ is positive for all points.
\end{proof}

\begin{proposition} \label{proposition free is fillable} $(T^2\times L_k,\ker \alpha_k), k\in\mathbb{N}$ is weakly symplectically fillable. 
  \end{proposition}
\begin{proof} The contact structure $(T^2\times S^3,\ker\alpha_1)$ given by (\ref{alpha}) is a particular example of the contact structure considered in Lemma \ref{general lemma} and since 
 $\alpha_k,$ $k>1$ is obtained from $\alpha_1$ under the diagonal $\mathbb{Z}_k$-action on $S^3,$ it follows that $(T^2\times L_k,\ker \alpha_k)$, $k>1$ is also a particular example considered in Lemma \ref{general lemma}. In order to apply Lemma \ref{general lemma} we have to show that $(S^3,\ker\alpha_1')$ and $(L_k,\ker\alpha_k'),$ $k>1$ are weakly fillable contact manifolds, where $\alpha_1'=\frac{i}{4}(z_1d\overline{z}_1-\overline{z}_1dz_1-(z_2d\overline{z}_2-\overline{z}_2dz_2))$ and $\alpha_k'$ denotes the $\mathbb{Z}_k$-quotient of $\alpha_1'$. 
\vskip1mm
 Note that $\ker\alpha_k',$ $k\in\mathbb{N},$ is the standard tight contact structure on the lens space $L_k,$ $(L_1=S^3)$ (see Sections 2.1. and 2.3. in \cite{Eliashberg1}), thus not overtwisted. Moreover, $(L_k,\ker\alpha_k')$ is invariant under the standard $T^2$-action on $L_k,$ thus, $(L_k,\ker\alpha_k')$ is a toric contact 3-manifold.  Since $L_k$ is not homeomorphic to $S^1\times S^2,$ according to classification of toric contact 3-manifolds, it follows that $(L_k,\ker\alpha_k')$ is of Reeb type, thus strongly fillable (Corollary \ref{corollary non free is fillable}) and, in particular, weakly fillable.
\end{proof}
We are now ready to prove Theorem \ref{theorem 3-dim} and Theorem \ref{theorem fillable}. As for toric contact 3-manifolds, we recall that Giroux in \cite{Giroux} proved that $(T^3,\xi_k=\ker(\cos( k\theta) d\theta_1+\sin( k\theta) d\theta_2)),$  $k\in\mathbb{N}$ are weakly fillable while Eliashberg in \cite{Eliashberg2} proved that only $(T^3,\xi_1)$ is strongly symplectically fillable. The rest of the proof of Theorem \ref{theorem 3-dim} follows from the classification of toric contact 3-manifolds (Section \ref{section classification}), Corollary \ref{corollary non free is fillable} and Proposition \ref{proposition non free is fillable}. The proof of Theorem \ref{theorem fillable} follows from
the classification (Section \ref{section classification}) and Corollary \ref{corollary non free is fillable}, Proposition \ref{proposition non free is fillable} and Proposition \ref{proposition free is fillable}.

\end{document}